\newif\ifspringer\springertrue
\springerfalse  
\ifspringer
  \documentclass[probth,numbook]{svjour}
  \usepackage{amsmath}
  \usepackage{times}
  \smartqed
  \let\svendproof=\endproof
  \def\endproof{\qed\svendproof}
\else
  \documentclass[11pt]{article}
  \usepackage{parskip}
  \usepackage{amsmath,amssymb,amsthm}
  \newtheorem{theorem}{Theorem}[section]
  \newtheorem{lemma}{Lemma}[section]
  \newtheorem{corollary}{Corollary}[section]

\newenvironment{acknowledgement}{\paragraph{Acknowledgements}}{}
  \newcommand{\keywords}[1]{\paragraph{Key words} #1}
  
  \let\qed=\bbox

  \makeatletter
  \@addtoreset{equation}{section}
  \makeatother
\fi
\usepackage[T1]{fontenc}
\newcommand\EE{\mathsf{E}}
\newcommand\II{\mathcal{I}}
\newcommand\PP{\mathsf{P}}
\newcommand\U{\mathcal{U}}
\newcommand\V{\mathcal{V}}
\def\DD{\displaystyle}

\def\egaldef{\stackrel{\mbox{\tiny def}}{=}}  
\def\wt{\widetilde}

\newcommand\1{\leavevmode\hbox{\rm \small1\kern-0.35em\normalsize1}}
\newcommand\ind[1]{\1_{\{#1\}}}
\newcommand{\limdec}{\mathop{\mathrm{lim}\scriptstyle\searrow}\limits}

\title{Birth and Death Processes on Certain Random Trees:
Classification and Stationary Laws} 
\ifspringer
  \titlerunning{Birth and Death Processes on Certain Random Trees}
  \author{Guy Fayolle \and Maxim Krikun \and Jean-Marc Lasgouttes
  \thanks{J.-M. Lasgouttes worked partly on the present study while
  spending a sabbatical at EURANDOM in Eindhoven.}} \institute{Guy
  Fayolle \at INRIA, Domaine de Voluceau, Rocquencourt BP 105, 78153
  Le Chesnay Cedex, France. \email{Guy.Fayolle@inria.fr} \and Maxim
  Krikun \at Laboratory of Large Random Systems, Faculty of
  Mathematics and Mechanics, Moscow State University, 119899, Moscow,
  Russia. \email{krikun@lbss.math.msu.su} \and Jean-Marc Lasgouttes
  \at INRIA, Domaine de Voluceau, Rocquencourt BP 105, 78153 Le
  Chesnay Cedex, France. \email{Jean-Marc.Lasgouttes@inria.fr}}
\else
  \author{
    Guy Fayolle\thanks{\texttt{Guy.Fayolle@inria.fr} -- INRIA -- Domaine de 
    Voluceau, Rocquencourt BP 105, 78153 Le Chesnay Cedex, France} 
  \and Maxim Krikun \thanks{\texttt{krikun@lbss.math.msu.su} --
    Laboratory of Large Random Systems -- Faculty of Mathematics and 
    Mechanics, Moscow State University, 119899, Moscow, Russia}
  \and Jean-Marc Lasgouttes\thanks{\texttt{Jean-Marc.Lasgouttes@inria.fr} --
    INRIA -- Domaine de Voluceau, Rocquencourt BP 105, 78153 Le
    Chesnay Cedex, France.} \thanks{J.-M. Lasgouttes worked partly
    on the present study while spending a sabbatical at EURANDOM in
    Eindhoven.}} 
\fi

\date{Received May 2002, revised August 2003}
\begin{document}
\maketitle

\begin{abstract}  
The main substance of the paper concerns the growth rate and the
classification (ergodicity, transience) of a family of random trees.
In the basic model, new edges appear according to a Poisson process of
parameter $\lambda$ and leaves can be deleted at a rate $\mu$. The
main results lay the stress on the famous number $e$. A complete
classification of the process is given in terms of the intensity
factor $\rho=\lambda/\mu\,$: it is ergodic if $\rho\leq e^{-1}$, and
transient if $\rho>e^{-1}$. There is a phase transition phenomenon:
the usual region of null recurrence (in the parameter space) here does
not exist. This fact is rare for countable Markov chains with
exponentially distributed jumps. Some basic stationary laws are
computed, e.g.~the number of vertices and the height. Various bounds,
limit laws and ergodic-like theorems are obtained, both for the
transient and ergodic regimes. In particular, when the system is
transient, the height of the tree grows linearly as the time
$t\to\infty$, at a rate which is explicitly computed. Some of the
results are extended to the so-called multiclass model.
\end{abstract}


\section{Introduction and model description}\label{NOTATION}

So far, very few results seem to exist for random trees as soon as
insertions and deletions are simultaneously permitted (see e.g.
\cite{MAH}). We shall study  one of the simplest models in this
class, which offers both interesting and non trivial properties.
Broadly speaking, one might think of a vertex as being a node of a
network (e.g.~the Internet) or of some general data structure. This
paper is a self-contained continuation of the report~\cite{FAY}.

Let $G=\{G(t),t\geq 0\}$ be a continuous time Markov chain with state
space the set of finite directed trees rooted at some fixed vertex
$v_{0}$. 

Throughout the study, the \emph{distance} between two vertices is the
number of edges in the path joining them, and the \emph{height} $h(v)$
of a vertex $v$ is the distance from the root. The set of vertices
having the same height $k$ form the $k$-th
\emph{level} of the tree, the root $v_{0}$ being at level $0$. Hence
the height of $G$ is a stochastic process $\{H_G(t),t\geq
0\}$, where
\[ 
H_G(t)\egaldef\max_{v\in G(t)} h(v).
\]
$N_G(t)$ will stand for the \emph{volume} of $G(t)$ (\emph{i.e.} its
total number of vertices).

Wherever the meaning is clear from the context, we shall omit the
subscript $G$ and simply write $H$ or $N$. The \emph{indegree} of a vertex
$v$ is the number of edges starting at $v$ and a vertex with indegree
$0$ is a \emph{leaf}. Finally, we will also need the classical notion
of \emph{subtree} with root $v$, which goes without saying.

At time $t=0$, $G(0)$ consists of the single vertex $v_{0}$. Then at
time $t>0$, the transitions on $G$ are of two types:

\begin{itemize} 
\item \textbf{Adjunction.} At each vertex $v$, a new edge having its
origin at $v$ can be appended to the tree at the epochs of a Poisson
process with parameter $\lambda>0$. In this case, the \emph{indegree}
of $v$ is increased by one and the new edge produces a new leaf.
\item \textbf{Deletion.} From its birth, a leaf (but the root) can
be deleted at a rate $\mu$. In other words, a vertex \emph{as long as
it has no descendant} has an exponentially distributed lifetime with
parameter $\mu \geq 0$.
\end{itemize}

\subsection{Organization of the paper, results and related studies} 
Section~\ref{DELETE} is devoted to the birth and death model described
above, with $\lambda, \mu>0$. An exact and complete classification of
$G$ is given. Indeed, necessary and sufficient conditions are derived
for the process to be ergodic ($\mu\ge\lambda e$) or transient
($\mu<\lambda e$). A phase transition phenomenon is enlightened, which
corresponds precisely to the absence of a null recurrence region.

When the system is ergodic, the stationary distributions of the volume and
 of the height of the tree are computed in
 section~\ref{DISTRIBUTIONS}.

Section~\ref{GROWTH} deals with limit laws and scalings for $H_G(t)$
and $N_G(t)$ in the transient case. The main outcome is a kind of
ergodic theorem for $H_G(t)$, valid for any $\mu\ge0$. It allows, in
the particular case $\mu=0$ (pure birth-process), to rediscover the
magic growth rate $\lambda e$, originally derived e.g.\ in
\cite{DEV,PIT}.

Finally, section~\ref{MULTI} proposes an extension to a multiclass
model, in which the parameters of the process depend possibly on the
class, the key result being a qualitative theorem for ergodicity.

Recently, the authors were made aware of a model studied
in~\cite{PUH1,LIG}. The setting considered there is a contact process
on a $d$-ary ordered tree, also known as a Catalan tree, where $d\geq
2$ is an arbitrary \emph{finite} integer. The main difference with our
model resides in the fact that each empty descendant of an occupied
vertex can become occupied at a rate $\beta>0$. The classification of
the process was obtained for $d=2$ in~\cite{PUH1}, and~\cite{LIG}
extends the result to any finite $d$, but nothing was said for
$d=\infty$. It turns out that most of the points presented in our
study (ergodicity, zero-one laws, etc.) cannot be obtained by simply
letting $d\to\infty$ in~\cite{PUH1,LIG}. Likewise, reversibility
arguments (which should theoretically lead to explicit invariant measures)
used in the latter papers do not seem to be effective when $d$ is
infinite (see section~\ref{DISTRIBUTIONS}).

\section{The  birth and death case: $\lambda>0,\mu>0$} 
\label{DELETE}
The random tree $G$ evolves according to the rules given in the
introduction, the first important question being to find exact
conditions for this process to be recurrent or transient. Main results
in this respect are stated in theorem~\ref{TH2}.

For convenience, we define the \emph{lifetime} $\tau_v$ of an
arbitrary vertex $v$, which measures the length of the time interval
between the birth and the death of $v$ (for consistency
$\tau_v=\infty$ if $v$ is never erased).

\begin{lemma}\label{SYS}
All vertices, but  the root, have the same lifetime distribution
$p(t)$, which satisfies the following system (S) 
\begin{eqnarray}
\beta(t) & = & \DD \mu\exp\Bigl\{-\lambda\int_0^t(1-p(x))dx\Bigr\},
                                                    \label{EQU1} \\[0.2cm] 
\beta(t) & = & \DD\frac{dp(t)}{dt} + \int_0^t \beta(t-y)dp(y), \label{EQU2}
\end{eqnarray}
with the initial condition $p(0)=0$.
\end{lemma}
\begin{proof} 
Let $v$ be a particular vertex of $G(t)$ and consider the related
random subtree with root $v$. Its evolution does not depend on
anything below $v$, as long as $v$ exists. Therefore all these
subtrees are identically distributed and, accordingly, their vertices
have the same lifetime distribution.

To capture more precisely the evolution of the process, we associate
with each vertex $v$ with age $t$ its number $X_v(t)$ of direct
descendants (i.e.\ who are located at a distance $1$ from $v$).

At rate $\lambda$, a vertex $v$ produces descendants whose lifetimes
are independent, with the common distribution $p(t)$. As soon as
$X_v(t)=0$, $v$ can die at rate $\mu$, in which case the process of
production stops.  

It is actually useful to extend $X_v(t)$ for all $t\ge0$ by deciding
that, instead of deleting $v$, a \emph{$\mu$-event} occurs without
stopping the production of descendants. With this convention, the
number of descendants of the root vertex $v_0$ evolves as
$X_{v_0}(t)$, for all $t\ge 0$. Let $\tau_v$ denote the random epoch
of the first $\mu$-event, which is distributed according to $p(t)$.

Clearly the process $X_v$ is regenerative with respect to the
$\mu$-events. Thus the random variables $X_v(t)$ and $X_v(\tau_v+t)$
have the same distribution.

For any fixed $t$, we write down a sum of conditional probabilities,
expressing the fact that $v$ had exactly $k$ descendants, who all have
died in $[0,t]$, their birth-times being independent and uniformly
spread over $[0,t]$. This yields at once equation (\ref{EQU1}),
since
\begin{eqnarray}
\PP\{X_v(t)=0\} 
   & = &\sum_{k=0}^\infty \frac{e^{-\lambda t}(\lambda t)^k}{k!} 
             \bigg(\int\limits_0^{t} \frac{p(x)dx}{t}\bigg)^k \nonumber \\ 
& = &\exp\Big\{-\lambda\int\limits_0^{t}(1-p(x))dx\Big\}. \label{PX0}
\end{eqnarray}

By means of a regenerative argument, it is also possible to rewrite
the above probability in another way, starting from the decomposition
\begin{equation}\label{DECOMP1}
\PP\{X_v(t)=0\}
  = \PP\{X_v(t)=0,\tau_v\ge t\} + \PP\{X_v(t)=0,\tau_v < t\}.
\end{equation}
In fact, we have the trite relations
\[ 
\begin{cases}
\DD\frac{dp(t)}{dt} = \mu\PP\{X_v(t)=0,\tau_v\ge t\},\\[0.3cm]
\PP\{X_v(t)=0,\tau_v < t\} = \PP\{X_v(t-\tau_v)=0,\tau_v < t\}, \\[0.3cm]
\PP\{\tau_v\in(y,y+dy)\} = dp(y) ,
\end{cases}
\]
which yield in particular,
\[
\PP\{X_v(t)=0,\tau_v < t\} = \int_0^{t}
\PP\{X_v(t - y)=0\}dp(y).
\]
Hence, putting $\beta(t)\egaldef\mu\PP\{X_v(t)=0\}$, one sees that
(\ref{DECOMP1}) corresponds term by term to (\ref{EQU2}). The proof
of the lemma is concluded.
\end{proof} 

It is convenient to introduce now $\tau\egaldef\tau_{v_0}$, which is
the random variable representing the epoch of the first $\mu$-event
for the root of the tree. Nonetheless in the sequel, especially in
sections \ref{DISTRIBUTIONS} and \ref{GROWTH}, $\tau$ will also often
refer to the lifetime of an arbitrary generic vertex $v$, owing to the
fact that all these quantities have the same distributions.

We are ready to state the main result of this section.

\begin{theorem}\label{TH2} \mbox{ }
\begin{itemize}
\item[(A)] The Markov chain $G$ is ergodic if, and only if,  
\begin{equation}\label{ERG}
\rho\egaldef\frac{\lambda}{\mu} \leq \frac{1}{e}\,.
\end{equation}
\item[(B)] When the system is ergodic, the mean lifetime
 $m\egaldef\EE \tau$ is given by 
\[
m = \frac{r}{\lambda},
\]
where $r\leq 1$ denotes the smallest root of the equation 
\begin{equation}\label{ROOT} 
 re^{-r}=\rho     
\end{equation}
and represents the mean number of descendants of an arbitrary vertex
at steady state.
\item[(C)] When $\rho > \DD \frac{1}{e}$, then the system is
transient. In this case, 
\[
  \lim_{t\to\infty} p(t)\egaldef \ell<1.
\] 
As a rule, $x$ being the positive root of $xe^x=\rho^{-1}$, we
have for any $\rho$
\[x \le \ell \le \min \Bigl(1,\frac{1}{\rho}\Bigr)\quad  \mathrm{and}\ 
\lim_{\rho\to\infty} \rho\ell =1 .
\]
\end{itemize}
\end{theorem}
The proof of the theorem is spread over the next two subsections.

\subsection{Ergodicity}\label{ERGODICITY} 
Relying on  standard theory of Markov chains with countable state
space (see~\cite[vol. I]{FEL}), we claim the system ergodic if, and
only if, $m<\infty$. As a matter of fact, the $\mu$-events are
regeneration points for the process $X(t)$, which represents exactly
the number of descendants of the \emph{root} $v_0$. Hence when
$\EE\tau<\infty$ (i.e. $\beta(\infty)>0$), the event $\{X(t)=0\}$
has a positive probability, so that $G$ is ergodic. Conversely, if $\EE
\tau=\infty$ then $X(t)$ is transient and so is $G$.

For an arbitrary positive function $f$, denote by $f^*$ its ordinary
Laplace transform
\[f^*(s)\egaldef \int_0^{\infty} e^{-st}f(t)dt, \quad \Re (s) \geq 0.
\]
Later on we will also need the associated inversion formula (see
e.g.~\cite{FUC})
\begin{equation}\label{LAP}
f(t) =\frac{1}{2i\pi} \int_{\sigma-i\infty}^{\sigma+i\infty}
e^{st}f^*(s)ds, \quad \Re (\sigma) >0.
\end{equation}
To show the necessity of condition (\ref{ERG}), we suppose $G$ is
ergodic. In this case, by~(\ref{EQU1}), the quantity $\DD
\lim_{t\to\infty} \beta(t)$ does exist and
\[\mu e^{-m}\leq\beta(t)\leq \mu ,\]
 so that we can apply a limiting relation of Abelian type (see
e.g.~\cite{FUC}). Hence equations (\ref{EQU1}) and (\ref{EQU2})---the
latter belonging to the Volterra class---yield respectively
\begin{equation}\label{NEC}
\begin{cases} \DD\lim_{t\to\infty}\beta(t) = \lim_{s\to 0} s\beta^*(s) = 
\lim_{s\to 0}\frac{s^2p^*(s)}{1-sp^*(s)} = \frac{1}{m} ,\\[0.3cm]
\DD\lim_{t\to\infty}\beta(t) = \mu e^{-\lambda m},
\end{cases}
\end{equation}
whence the equality
\[
\rho = \lambda m e^{-\lambda m}.
\]
As the function $xe^{-x}$ reaches its maximum $e^{-1}$ at $x=1$, we
conclude that necessarily $\rho\leq e^{-1}$.

In order to prove the sufficiency of (\ref{ERG}), we have to get a
deeper insight into system (S). There will be done along two main
steps.

\paragraph{(a)} Although (S) reduces to a second order nonlinear
integro-diffe\-rential equation, this does not help much. What is more
useful is that all derivatives $p^{(n)}(0)$, $\beta^{(n)}(0)$, taken at
the the origin in the complex $t$-plane, can be recursively computed
for all $n$. This can be checked at once, rewriting (\ref{EQU1}) in
the differential form
\begin{equation}\label{EQU3}
\frac{d\beta(t)}{dt}+ \lambda (1-p(t))\beta(t)=0.
\end{equation}  
Noticing the derivatives $p^{(n)}(0)$ and $\beta^{(n)}(0)$
have alternate signs when $n$ varies, it is direct to verify that
$\beta$ and $p$ are analytic functions around the origin, and that
their respective power series have a non-zero radius of convergence.
The first singularities of $p$ and $\beta$ are on the negative real
axis, but not easy to locate precisely. Thus (S) has a solution,
which is unique, remarking also that uniqueness is a mere consequence
of the Lipschitz character of $dp(t)/dt$ with respect to $\beta$ in
the Volterra integral equation (\ref{EQU2}) (see e.g.~\cite{CAR}). En
passant, it is worth noting that the solution in the whole complex
plane---which is not really needed for our purpose---could be obtained
by analytic continuation directly on system (S).

\paragraph{(b)} When (\ref{ERG}) holds, the next stage consists in
exhibiting a \emph{non-defective} probabilistic solution $p(t)$
[necessarily unique by step (a)], with a finite mean
$m<\infty$.  This is more intricate and will be achieved by
constructing a converging iterative scheme.

Consider the system 
\begin{equation}\label{ITE}
\begin{cases}
 \beta_{0}(t) & = \ \mu, \quad t\ge 0 \,, \\[0.2cm] \beta_k(t) & = \
 \DD\frac{dp_{k}(t)}{dt} + \int_0^t \beta_{k}(t-y) dp_{k}(y),
 \\[0.3cm] \beta_{k+1} (t) & = \
 \DD\mu\exp\Bigl\{-\lambda\int_0^t\bigl(1-p_k(y)\bigr)dy\Bigr\},\\[0.3cm]
 p_k(0) & = \ 0, \ \forall k \ge 0\,.
\end{cases}
\end{equation}
 The second equation in (\ref{ITE}) is equivalent to
\begin{equation}\label{CONV}
sp_k^*(s)= \frac{\beta_k^*(s)}{1+\beta_k^*(s)},
\end{equation}
allowing to derive $p_k$ from $\beta_k$ by means of (\ref{LAP}) (see
also~\cite{FEL} for various inversion formulas in the real plane).
Then the computational algorithm becomes simple:
\begin{enumerate}
\item $p_0(t)= 1-e^{-\mu t}$.

\item Compute $\beta_1(t)= \mu\exp\bigl[-\rho (1-e^{-\mu
t})\bigr]$.

\item Compute $p_1(t)$, then $\beta_2(t), p_2(t)$, etc.
\end{enumerate}
In the scheme (\ref{ITE}), the initial condition $\beta_0(t)=\mu$ is
tantamount to take an implicit fictitious function, say $p_{-1}$,
satisfying $p_{-1}(t)=1,\ \forall t\ge 0$.

At each step, the successive $p_k$'s are non-defective probability
distributions, with finite means denoted by $m_k$.  The scheme
(\ref{ITE}) enjoys two nice properties.

\textbf{(i)} It is monotone \emph{decreasing}. Suppose $p_k(t)\leq
p_{k-1}(t)$, which is in particular true for $k=1$. In the third
equation of (\ref{ITE}), $\beta_{k+1}(t)/\mu$ is simply the
probability of being empty for an \textsc{m/g/$\infty$} queue with
arrival rate $\lambda$ and service time distribution function~$p_k$.
It is therefore possible, by a coupling argument, to build two
\textsc{m/g/$\infty$} queues corresponding to $\beta_k(t)$ and
$\beta_{k+1}(t)$ such that the $\mu$-event pertaining to level $k+1$
will always occur later than the one for level $k$. Thus
$p_{k+1}(t)\leq p_k(t)$.

So, the positive sequences $\{p_k(t), \beta_k(t), k\geq 0\}$ are
uniformly bounded and non-increasing for each fixed $t$. Consequently,
\[
p(t) = \limdec_{k\to\infty} p_k(t) \quad \textrm{and} \quad
\beta(t) = \limdec_{k\to\infty} \beta_k(t)
\]
form the unique solutions of (S). 

\textbf{(ii)} Letting $r_k\egaldef\lambda m_k$ and combining the two main
equations of (\ref{ITE}), we get
\[
r_{k+1}=\rho e^{r_k}, \ \forall k\ge 0, \quad \textrm{with} \
r_0=\rho.
\]
When $\rho\leq e^{-1}$, the $r_k$'s form an increasing sequence of
positive real numbers, with a finite positive limit $r$ satisfying
equation (\ref{ROOT}). Since $1- p_k(t)$ is also an increasing
sequence of positive functions, the theorem of Beppo Levi ensures the
equality
\begin{equation}\label{LIM}
\int_{0}^{\infty} (1-p(t))dt = \lim_{k\to\infty}\int_{0}^{\infty}
(1-p_k(t))dt = \lim_{k\to\infty} m_k = \frac{r}{\lambda}.
\end{equation}
It is worth to point out that (\ref{ITE}) is equivalent to the
construction of a sequence of trees $\{G_k,k\ge 0\}$, such that, for
any finite $k$, $G_k$ is ergodic and has a height not greater than
$k$.

This completes the proof of points (A) and (B) of the theorem. \qed

\subsection{Transience}\label{TRANSIENCE}
It turns out that the classification of the process for $\rho>e^{-1}$
can be obtained  from analytic arguments.

Recalling that $\DD \ell =  \lim_{t\to\infty} p(t)$, we define
\begin{equation}\label{TRANS0}
\DD\varepsilon(t)\egaldef\lambda\int_0^t(\ell-p(x))dx, \quad 
\lim_{t\to\infty} \uparrow \varepsilon(t)\egaldef \bar{\varepsilon},
\end{equation}
and it will be convenient to write $\varepsilon'(t)\egaldef
\frac{d\varepsilon(t)}{dt}$.

The quantity $\bar{\varepsilon}$ exists, but a priori is not
necessarily finite. It represents the mean number of descendants of an
arbitrary vertex $v$, conditioning on the fact that $v$ is the root of
an almost surely finite tree. In fact we will show that if
$\rho>e^{-1} $ then $\ell<1$, in which case the system is transient.
Thus there is no null-recurrence region in the parameter space.
However, $\bar{\varepsilon}$ given by (\ref{TRANS0}) will appear to be
\emph{finite for all $\rho>0$}.

Unless otherwise stated, $s$ in this subsection will stand for a
positive real variable. Then by a direct computation we get
\[
\beta^*(s) =\mu \int_0^\infty \exp
\bigl[-\bigl(\varepsilon(t)+(\lambda(1-\ell)+s)t\bigr)\bigr] dt ,
\]
together with the functional equation
\begin{equation}\label{TRANS}
\frac{s^2\varepsilon^*(s)}{\lambda} = \frac{\ell +
(\ell-1)\beta^*(s)}{1+ \beta^*(s)}. 
\end{equation} 

\textbf{(i)} \emph{Suppose now for a while $\ell=1$}. Then
(\ref{TRANS}) reduces to
\begin{equation}\label{TRANS1}
 s^2\varepsilon^*(s)\bigl[1 + \beta^*(s)\bigr] = \lambda ,
\end{equation}
and the idea is to show that, for $\rho>e^{-1}$, (\ref{TRANS1})
\emph{has no admissible solution, i.e. a solution such that
$\limdec_{t\to\infty}\varepsilon'(t)=0$}.

By well known theorems for Laplace transforms (see e.g. \cite{FUC}),
we have the relations
\begin{equation*}
\lim_{s\to 0}s\varepsilon^*(s) = \bar{\varepsilon},\quad
\lim_{s\to 0}s^2\varepsilon^*(s) = 0, \quad
\lim_{s\to 0}s\beta^*(s) = \mu\exp (-\bar{\varepsilon}).
\end{equation*}
When the system is ergodic, $\bar{\varepsilon}<\infty$ and the above
 limit equations give at once \begin{equation}\label{TAUBER}\lim_{s\to
 0}s^2\varepsilon^*(s)\beta^*(s)=\mu 
 \bar{\varepsilon}\exp(-\bar{\varepsilon}).
\end{equation} 
In the case $\bar{\varepsilon} =\infty$, the question is more
difficult and (\ref{TAUBER}) does not hold without additional conditions
on $\varepsilon(t)$, as for instance slow variation (see tauberian
theorems in~\cite{FEL}). The only cheap by-product of (\ref{TRANS1})
is the existence of the decreasing limit
\[\lim_{s\to0} \uparrow s^2\varepsilon^*(s)\beta^*(s) = 
\lambda - \limdec_{s\to0}s^2\varepsilon^*(s).
\]
To get deeper insight into (\ref{TRANS1}), we remark that the quantity
$s\varepsilon^*(s)\beta^*(s)$ can be viewed as the Laplace transform
of a convolution measure with density
\[
\int_0^t\varepsilon'(z)\exp[-\varepsilon(t-z)]dz,
\]
so that  (\ref{TRANS1}) is equivalent to the integro-differential equation
\begin{equation} \label{TRANS2}
\rho = \frac{1}{\mu} \varepsilon'(t)+ \int_0^t\varepsilon'(z) 
\exp[-\varepsilon(t-z)]dz.
\end{equation}
It is worth remembering that we are searching for solutions of
 (\ref{TRANS2}) in the class of functions $\varepsilon(t)$ which have
 positive monotone decreasing derivatives [this property can in fact
 be established by taking derivatives of higher order in
 (\ref{TRANS2})], and satisfy the intial condition $\varepsilon(0)=0$.
 Hence, as $t\to\infty$, the limit of the integral in (\ref{TRANS2})
 exists and \[ \lim_{t\to\infty} \int_0^t
 \varepsilon'(z)\exp[-\varepsilon(t-z)]dz = \rho -
 \limdec_{t\to\infty} \frac{\varepsilon'(t)}{\mu} >0. \] Choose $T,
 0<t<T$. Then the decomposition
\[
 \rho = \frac{1}{\mu} \varepsilon'(T) +
 \int_0^t\varepsilon'(z)\exp[-\varepsilon(T-z)]dz +
 \int_t^T\varepsilon'(z)\exp[-\varepsilon(T-z)]dz
\]
yields by monotonicity
\begin{equation} \label{TAUBER1}
\rho \le \frac{1}{\mu} \varepsilon'(T) + \varepsilon(t) 
\exp[-\varepsilon(T-t)] +
\int_t^T\varepsilon'(z)\exp[-\varepsilon(T-z)]dz.
\end{equation}
Putting $T=2t$ in (\ref{TAUBER1}) and using $\varepsilon'(2t)
\le\varepsilon'(t)$, we obtain the main inequality
\begin{equation} \label{TAUBER2} 
\rho - \varepsilon(t)\exp[-\varepsilon(t)] \le 
\varepsilon'(t) \biggl[ \frac{1}{\mu} +
\int_0^t\exp[-\varepsilon(z)]dz\biggr].
\end{equation}
On the other hand, (\ref{TRANS2}) shows immediately that the
right-hand side member of (\ref{TAUBER2}) is bounded by $\rho$.
Finally, any solution of (\ref{TRANS2}) must satisfy
\begin{equation} \label{TRANS3}
\rho -  \varepsilon(t)\exp[-\varepsilon(t)] \le \varepsilon'(t) 
\biggl[ \frac{1}{\mu} +
\int_0^t\exp[-\varepsilon(z)]dz\biggr] \le \rho.
\end{equation}
\emph{Assume also now}  $\rho>e^{-1}$. Then  
\[\sup_{t\ge0} \bigl[\rho -  \varepsilon(t) 
\exp[-\varepsilon(t)]\bigr]\egaldef\xi>0,
\]
 so that (\ref{TRANS3}) implies
\begin{equation} \label{TRANS4}
\xi \le \varepsilon'(t) \biggl[ \frac{1}{\mu} +
\int_0^t\exp[-\varepsilon(z)]dz\biggr] \le \rho.
\end{equation}
Let $a\in[\xi,\rho]$ be a real parameter and consider the differential
equation
\begin{equation} \label{TRANS5}
f'(t) \biggl[ \frac{1}{\mu} +
\int_0^t\exp[-f(z)]dz\biggr] =  a.
\end{equation}
The next step is to show that all solutions of (\ref{TRANS4}) are
located in the \emph{strip} delimited by the maximal and minimal
solutions of (\ref{TRANS5}). Setting $g(t)= \exp[-f(t)]$, a simple
calculation gives
\begin{equation} \label{DIFF}
\begin{cases}
\DD g'(t)= -a \log \bigl(1+\mu g(t)\bigr) +1\ge 0, \\[0.2cm]
\DD \int_0^{g(t)} \frac{dy}{1- a\log (1+\mu y)} = t.
\end{cases}
\end{equation}
The first equation of (\ref{DIFF}) yields $1+\mu g(t) \le\DD
e^{\frac{1}{a}}$, whence $\sup_{t\ge 0} g(t) <\infty$. On the other
hand, the second equation of (\ref{DIFF}) implies, for any fixed $t$,
that $g(t)$ (resp. $f'(t)$) is decreasing (resp. increasing) with
respect to $a$. Thus the maximal and minimal solutions of
(\ref{TRANS5}) take place respectively for $a=\rho$ and $a=\xi$, and
we have
\[
 \inf_{t\ge 0}f'(t)\ge \xi \exp [\xi^{-1}], \quad \forall a
 \in[\xi,\rho]. \] Consequently, by (\ref{TRANS4}),
 $\limdec_{t\to\infty}\varepsilon'(t)>0$, which contradicts
 null-recurrence, and shows finally that if $\rho>e^{-1}$ then
 necessarily $\ell<1$ (the announced transience).

\textbf{(ii)} Consider now the case $\ell<1$. 
Letting $s\to 0$ in (\ref{TRANS}) yields immediately
\[
\beta^*(0)= \frac{\ell}{1-\ell}.
\]
Then rewriting (\ref{TRANS}) in the form
\begin{equation}\label{TRANSbis}
\frac{s\varepsilon^*(s)}{\lambda} = \frac{(1-\ell)[\beta^*(0) - 
\beta^*(s)]}{s+ s\beta^*(s)}, 
\end{equation} 
letting again $s\to 0$ in (\ref{TRANSbis}) and using l'HÙpital's rule,
we obtain
\[
\bar{\varepsilon} = \lambda(1-\ell)^2 \int_0^\infty t\beta(t)dt <\infty .
\]

The exact computation of $\ell$ proves to be a difficult project.
Actually, since one can hardly expect more than approximate formulas,
we shall present various results, both formal and concrete, some of
them yielding bounds for $\ell$.

\subsubsection{Formal approach} 
Using the definition (\ref{TRANS0}), it appears that the right-hand
side member of (\ref{TRANS}) can be analytically continued to the
region $\Re (s) < -\lambda(1-\ell)$. Thus an analysis of singularities
becomes theoretically possible, which should hopefully allow to
compute $\ell$. We roughly sketch the method, without giving an
exhaustive presentation of all technicalities.

Owing to the inversion formula (\ref{LAP}), we can rewrite
(\ref{TRANS}) in the functional form
\begin{equation}\label{FUNC}
\frac{1}{2i\pi} \int_{\sigma-i\infty}^{\sigma+i\infty}
e^{st}\beta^*(s)ds = \mu\exp \biggl[\frac{-\lambda}{2i\pi}
\int_{\sigma-i\infty}^{\sigma+i\infty}\frac{e^{st}ds}{s^2\bigl(1+\beta^*(s)
\bigr)}\biggr],\quad \Re (\sigma) >0.  
\end{equation}
Arguing by analytic continuation in (\ref{FUNC}), it is possible to
prove that $\beta^*(s)$ is a meromorphic function with real negative
poles. Hence, $\beta(t)$ can be represented by the Dirichlet series
\begin{equation}\label{BETA}
\beta(t) = C\exp \biggl[\frac{-\lambda t}{1+\beta^*(0)}\biggr] +
\sum_{i\ge0}u_i e^{-\sigma_i t},
\end{equation}
where $C$ is a constant, the $\sigma_i$'s form a sequence of positive
increasing numbers satisfying
\[
\sigma_i> \frac{\lambda}{1+\beta^*(0)}, \quad \forall i\ge 0,
\]
and the $u_i$'s are \emph{ad hoc} residues. In the ergodic case
$\beta^*(0)=\infty$ and the first term in (\ref{BETA}) reduces to the
constant $C$. Then, $\varepsilon(t)$ could be obtained by formal
inversion of $\beta^*(s)$. Alas, the computation becomes formidable
and we did not get an exact tractable form (if any at all~!) for
$\ell$, since this is equivalent to compute $u_i, \sigma_i,i\ge0$.

\subsubsection{Bounds and tail distribution}
Beforehand, it is worth quoting some simple facts. First, the value of
$\ell$ does solely depend on $\rho$, as can be seen by scaling in
system (\ref{EQU1}--\ref{EQU2}), with the new functions
\[
 \widetilde{\beta}(t)=\frac{1}{\mu}\beta\Bigl(\frac{t}{\mu}\Bigr),
 \quad \widetilde{p}(t)=p\Bigl(\frac{t}{\mu}\Bigr).
\]
Secondly, combining  (\ref{EQU1}) and (\ref{EQU2}) leads to
the inequality
\[
 \beta(t) \le \mu - \lambda p(t),
\]
which yields
\begin{equation} \label{LBOUND}
\ell \le \min \Bigl(1,\frac{1}{\rho}\Bigr), \quad \forall \rho < \infty .
\end{equation}

In Section~\ref{ERGODICITY}, we also could have considered the scheme
\begin{equation}\label{ITE1}
\begin{cases}
 \gamma_{0}(t) & = \  \mu e^{-\lambda t} , \quad t\ge 0 \,, \\[0.2cm]
 \gamma_{k}(t) & = \ \DD\frac{dq_{k}(t)}{dt} + \int_0^t
 \gamma_{k}(t-y) dq_{k}(y), \\[0.3cm] \gamma_{k+1} (t) & = \
 \DD\mu\exp\Bigl\{-\lambda\int_0^t\bigl(1-q_k(y)\bigr)dy\Bigr\},\\[0.3cm] 
  q_k(0) & = \ 0, \ \forall k \ge 0\,,
\end{cases}
\end{equation}
which differs from (\ref{ITE}) only by its first equation, but this is
a crucial difference, corresponding in some sense to a fictitious
function $q_{-1}(t)=0,\ \forall t\ge 0$. Actually, this scheme
produces a sequence of trees $\{L_k, k\ge 0\}$, with the property that
the leaves of $L_k$ at level $k$ never die. Its basic properties are
the following:
\begin{itemize}
\item the $q_k$'s form an \emph{increasing} sequence of
\emph{defective}  distributions;
\item for all $k\ge0$, the tail distribution of $q_k$ dominates a
defective exponential distribution with  density of the form
$a_kb_ke^{-b_kt}$. Moreover, under
condition (\ref{ERG}), we have
\[
  \lim_{k\to\infty}a_k = 1,\quad\lim_{k\to\infty}b_k = \frac{\lambda}{r}
\]
and $q_k$ converges in $L_1$ to the proper distribution $p$. 
\end{itemize}

The iterative scheme (\ref{ITE1}) is convergent
\emph{for all} $\rho$, but  the distributions $q_k(t)$, $k\ge 0$,
are \emph{defective}, their limit being proper if and only if
$\rho\leq e^{-1}$. When $\rho>e^{-1}$, the limiting function $p(t)$
remains defective and
\[
 \lim_{t\to\infty}p(t)= \lim_{k\to\infty}\lim_{t\to\infty}q_k(t)=\ell <1.
\]
We shall derive bounds on $\ell$, in showing by induction that
$q_k(t)$, for $t$ sufficiently large, dominates an exponential
distribution. The idea of  proof will appear from the very first
step $k=1$. Actually, we have
\[ 
\begin{cases} 
q_0(t)= \ell_0(1-e^{-\theta_0t}) , \\[0.2cm]
\DD \ell_0
  = \frac{\mu}{\lambda+\mu}, \quad \theta_0=\lambda+\mu , \\[0.2cm] 
\DD \gamma_1(t) = \mu \exp \Bigl[-\lambda\bigl(1-\ell_0\bigr)t +
  \frac{\lambda \ell_0}{\theta_0}\big(e^{-\theta_0t}-1\bigr)\Bigr],
\end{cases} 
\]
and the Laplace transform $\gamma^*_1(s)$ has an explicit form,
based on the formula (which involves the incomplete gamma function,
see e.g.~\cite{GR})
\begin{equation} \label{eq:gamma}
  \II(x,y)=\int_0^\infty \exp \Bigl[-xt+ye^{-t}\Bigr] dt=
  \sum_{n=0}^\infty\frac{y^n}{n!(x+n)} , \quad \Re (x) >0. 
\end{equation}
By scaling, for any  constant $c>0$, we have
\begin{equation}\label{SCALE}
\frac{1}{c}\II\bigg(\frac{x}{c},y\bigg)=
\int_0^\infty \exp \Bigl[-xt+ye^{-ct}\Bigr] dt ,
\end{equation}
 so that
\[ 
\gamma^*_1(s)= \frac{\mu}{\theta_0} \exp \biggl(\frac{-\lambda
\ell_0}{\theta_0}\biggr) \II\biggl[\frac{s +\lambda
(1-\ell_0)}{\theta_0}, \frac{\lambda \ell_0}{\theta_0}\biggr], \quad
  \Re \bigl(s+ \lambda (1-\ell_0)\bigr)> 0.
\] 
 The series in equation (\ref{eq:gamma}) shows that $\gamma^*_1(s)$ can
be analytically continued as a meromorphic function of $s$, with
simple poles $s_n= -\lambda (1-\ell_0)-n, n\ge 0$.

Similarly, one checks easily  the roots in $s$ of
$\gamma^*_1(s)+1=0$ are simple, real and negative. Denoting them by
$-z_n, n\ge 0$, we have the following 

\begin{lemma}\label{LEVEL1}
\[ 
q_1^*(s)= \frac{\gamma_1^*(s)}{s(1+\gamma_1^*(s))}
\]
is a meromorphic function of $s$, with poles at $0,-z_0, -z_1,\ldots$, where
\[
\lambda (1-\ell_0) + n\theta_0 < z_n < \lambda (1-\ell_0) +
 (n+1)\theta_0, \quad n\ge 0, 
\]
with  the more precise bounds
\begin{equation}\label{ZERO}
\frac{\mu \theta_0}{\mu + \theta_0}\exp
\biggl(\frac{-\lambda\ell_0}{\theta_0}\biggr) \le z_0 -\lambda
(1-\ell_0)\le \min \biggl[\theta_0, \mu
 \exp\biggl(\frac{-\lambda\ell_0}{\theta_0}\biggr) \biggr].
 \end{equation} 
Hence
\begin{equation} \label{STEP1}
q_1(t) =  \ell_1 - \sum_{n\ge 0} r_n e^{-z_nt},
\end{equation}
where the residue $r_n$ of $q_1^*(s) e^{st}$ at the pole $z_n, n\ge 0$,
is positive and given by the linear relation
\[ 
 r_n z_n \frac{d\gamma^*_1}{ds}_{\mid s=-z_n} + 1 = 0,
\]
and 
\[
  \ell_1
 = \sum_{n\ge 0} r_n = \frac{\gamma^*_1(0)}{1+\gamma^*_1(0)}.
\]
Moreover, (\ref{STEP1}) yields
\begin{equation} \label{INDUC}
\ell_1 (1- e^{-z_0\, t}) \le  q_1(t) \le \ell_1. 
\end{equation}
\end{lemma}

\begin{proof} Only the first part of (\ref{ZERO}) needs
some explanation. It is obtained by checking that, for all $y\ge0$, the
first negative root in $x$ of the equation
\[
\mu x\exp (-y) \II(x, y) +\theta x = 0, \ y\ge 0,
\]
 satisfies
\[ 
 \begin{cases} x\theta_0 \ge - \mu e^{-y} \\[0.3cm]
\DD x^2+\bigl(1+\frac{\mu}{\theta_0}\bigr) + \frac{\mu}{\theta_0}e^{-y}\le 0.
 \end{cases}
\]
\end{proof}

We shall prove by induction that $p(t)$ dominates a \emph{reasonable}
exponential distribution. To this end, assume
\[ \ell_k (1- e^{-\theta_k\, t}) \le  q_k(t), 
\]
which is in particular true for $k=0,1$, as shown in lemma
\ref{LEVEL1}. Then the calculus which led to (\ref{INDUC}) yields also
\[ \ell_{k+1} (1- e^{-\theta_{k+1}\, t}) \le q_{k+1}(t),
\]
where $(\ell_{k+1},\theta_{k+1})$ can be derived from
$(\ell_k,\theta_k)$ by the formulas
\begin{equation}\label{LBOUND1}
\begin{cases}
\DD \alpha_k  =   \frac{\mu}{\theta_k} \exp \biggl(\frac{-\lambda
\ell_k}{\theta_k}\biggr) \II\biggl[\frac{\lambda
  (1-\ell_k)}{\theta_k},\frac{\lambda \ell_k}{\theta_k}\biggr], \\[0.5cm]
\DD \ell_{k+1}  =  \frac{\alpha_k}{1+\alpha_k}.
\end{cases}
\end{equation}
and $\theta_{k+1}$ is the first positive root of the equation
\begin{equation}\label{LBOUND2}
\frac{\mu}{\theta_k} \exp \biggl(\frac{ -\lambda
\ell_k}{\theta_k}\biggr) \II\biggl[\frac{-\theta_{k+1}+\lambda
  (1-\ell_k)}{\theta_k},\frac{\lambda \ell_k}{\theta_k}\biggr] + 1 = 0.
\end{equation}
 Replacing $\theta_0$ and $z_0$ in (\ref{ZERO}) by $\theta_k$  and
 $\theta_{k+1}$ respectively, one can prove the existence of
\[
\lim_{k\to\infty}(\ell_k,\theta_k)= (\ell_d,\theta), \quad
 \theta < \infty, \ \ell_d \le \ell \le 1,
\]
where $0<\theta <\infty$ when $\ell<1$. 

The numerical computation of $\ell_d$ is freakish in the ergodicity
region (where the determination of the $\theta_k$'s is a source of
numerical instability), but proves very satisfactory for $\rho\gg
e^{-1}$.

Next, instead of providing as in lemma \ref{LEVEL1} a stochastic
ordering for all $t$, we get a tail-ordering, which has the advantage
of achieving the \emph{exact} value $\ell=1$ for $\rho \le e^{-1}$.

\begin{lemma}\label{LEVEL2}
 \begin{equation}\label{IND} q_k(t) \ge a_k (1-
 e^{-b_kt})+o(e^{-b_kt}), \quad \forall\ k\ge 0,
\end{equation} 
where the sequence $(a_k, b_k)$ satisfies the recursive scheme
\begin{equation}\label{REC}
\begin{cases}
 \DD a_{k+1}b_{k+1}= \mu \exp\Bigl(\frac{-\lambda a_k}{b_k}\Bigr),
 \\[0.5cm] 
\DD b_{k+1}(1-a_{k+1})= \lambda (1-a_k), 
\end{cases}
 \end{equation} 
with $a_0 = \mu/(\lambda+\mu)$ and $b_0 = \lambda+\mu$. 

 Setting $\DD a\egaldef\lim_{k\to\infty}a_k$ and $\DD
 b\egaldef\lim_{k\to\infty}b_k$ in (\ref{REC}), one has the limits
 \begin{equation}\label{AB} 
\begin{cases} 
\DD a=1,\  b=\frac{\lambda}{r}, \ \mathrm{if} \quad \rho \le
 e^{-1},\\[0.2cm] a=x, \ b= \lambda, \ \mathrm{if} \quad \rho \ge
e^{-1},
\end{cases}
 \end{equation} 
where $x\le 1$ is the root of the equation
 \begin{equation}\label{ROOT2} xe^x= \frac{1}{\rho}. 
\end{equation}
\end{lemma}

In the course of the proof of lemma~\ref{LEVEL2}, we will have to
characterize positive measures when then are defined from a Laplace
transform of the form $\frac{f^*(s)}{1+f^*(s)}$, as for instance in
(\ref{CONV}), which a priori does not correspond to a \emph{completely
monotone} function, according to the classical definition
of~\cite{FEL}. The following lemma does address this question and
might be of intrinsic interest.

\begin{lemma}\label{MONO}
Let $Q$ be a measure concentrated on $[0,\infty[$, and its
corresponding Laplace transform
$Q^{*}(s)=\int_{0}^{\infty}e^{-st}dQ(t)$, for any complex $s$ with
$\Re(s)\ge 0$. Define
\begin{eqnarray*}
\psi(s) &\egaldef& \int_0^\infty \mu e^{-(\lambda Q(t) +
st)}dt, \\
\omega(s)&\egaldef&\frac{\psi(s)}{1 +\psi(s)}.
\end{eqnarray*}

Then $\omega(s+\mu)$ is the Laplace transform of a positive measure
$\Delta_Q$ $[0,\infty[$. In addition $\Delta_Q$ is a decreasing
functional of $Q$, in the sense that, for all $ R\ge Q$, $R$ being
$Q$-continuous,
\[\Delta_R\le\Delta_Q.
\]
\end{lemma}
\begin{proof}
For any complex number $s$ with $\Re (s)\ge 0$,  $\widetilde{\psi}(s) \egaldef
\psi(s+\mu)$ can be viewed as the the Laplace transform of 
a positive measure $U$ having the density $\mu e^{-(\lambda Q(t)+\mu
t)}$. Thus
\begin{equation}\label{CONT}
 \PP\{U\leq t\} = \int_0^t\mu e^{-(\lambda Q(t)+\mu t)}dt \leq
 1-e^{-\mu t}\leq 1,
\end{equation}
and the following expansion holds
\begin{equation}\label{CONT1} 
  \omega(s+\mu)= \sum_{k=0}^\infty (-1)^k\widetilde{\psi}^{k+1}(s),
\end{equation}
where $\widetilde{\psi}^k(s)$ stands for the transform of the $k$-fold
convolution of $U$ defined in (\ref{CONT}). A function being uniquely
determined---up to values in a set of measure zero---by the values of
its Laplace transform in the region $\Re (s)\ge\mu$, the first part of
the lemma is proved. As for the monotony, one can differentiate the
inverse of (\ref{CONT1}) term by term (with respect to $Q$): since
each term is multiplied by $(-\lambda)^k$, the resulting series is
negative and the conclusion follows.
\end{proof}

\begin{proof}[\ifspringer\else Proof \fi of lemma~\ref{LEVEL2}]
Most of the ingredients reside in the integral representation of
$\gamma^*_1(s)$ by means of formula (\ref{eq:gamma}), and we shall
present the main lines of argument.

Fix a number $D,\,b_0<D<\infty$. Then (\ref{SCALE}) yields the inequality
\begin{equation}\label{PSI1}
\gamma^*_1(s) \ge \psi_1(s) \egaldef \frac{\mu}{D} \exp \biggl(\frac{-\lambda
a_0}{b_0}\biggr) \II\biggl[\frac{s+\lambda
  (1-a_0)}{D}, \frac{\lambda a_0}{b_0}\biggl],
\end{equation}
 whence
\[
sq_1^*(s)= \frac{\gamma_1^*(s)}{1+\gamma_1^*(s)}\ge 
\frac{\psi_1(s)}{1+\psi_1(s)}, \quad \forall s>0.
\]
Setting
\begin{equation}\label{PSI2}
u(s) = \mu\exp \biggl(\frac{-\lambda a_0}{b_0}\biggr) \frac{1}{\lambda
(1-a_0) + s},
 \end{equation}
and isolating the first term is the power series expansion of (\ref{PSI1})
by means of (\ref{eq:gamma}), we obtain after a routine algebra
\begin{equation}\label{PSI3}
\frac{\psi_1(s)}{1+\psi_1(s)} = \frac{u(s)}{u(s)+1} + 
\frac{w(s,D)}{D}.
\end{equation}
 Remarking the first pole of $w(s,D)$ is the root of $u(s)+1=0$, we
 can use lemma~\ref{MONO} and a term by term inversion of
 (\ref{PSI3}) to obtain
\[
q_1(t) \ge a_1(1-e^{-b_1 t}) + \frac{\mathcal{O}(e^{-b_1 t})}{D} ,
\]
where the couples $(a_1,b_1),(a_0,b_0)$ satisfy
system (\ref{REC}). 

At step $k=2$, one would introduce a constant, say $D_{1}$, and repeat
the same procedure to obtain (\ref{IND}). It might be useful to note
that it is not possible to take $D=\infty$, since this would create an
atom at $t=0$, in which case lemma~\ref{MONO} does not work in
general.
 
At last, it is a simple exercise (therefore omitted) to verify the
existence of $(a,b) =\lim_{k\to\infty} (a_k,b_k)$, given by
(\ref{AB}).  The proof of the lemma is terminated.
\end{proof}

Now, to conclude the proof of theorem~\ref{TH2}, it merely suffices to
note that, by (\ref{ROOT2}),
\[
\rho^{-1}- \rho^{-2}\le x \le \ell \le \rho^{-1}, \quad \forall \rho \ge
 e^{-1}.
\] 
\qed
\paragraph{Subsidiary comments} The method of schemes to analyze nonlinear
operators in a probabilistic context is extremely powerful (see
e.g.~\cite{DEL} for problems related to systems in thermodynamical
limit), and in some sense deeply related to the construction of
Lyapounov functions. Up to sharp technicalities, the schemes
(\ref{ITE}) and (\ref{ITE1}) can be exploited to derive precise
information about the speed of convergence as $t\to\infty$, for any
$\rho,\, 0<\rho<\infty$, and when pushing exact computations slightly
farther, one perceives underlying relationships with intricate
continued fractions. Finally, we note that the question of transience
could be studied from a large deviation point of view, by considering
$\varepsilon(t)$ as the member of a family indexed by the parameter
$(\rho-e^{-1})$---see in this respect section~\ref{GROWTH}.

\section{Some stationary distributions}\label{DISTRIBUTIONS}

In this section, we derive the stationary laws of some performance
measures of interest when the system is ergodic, i.e. $\rho\leq
e^{-1}$. Incidentally, note that the only process  studied
so far, that is the number $X$ of vertices attached to the root,
behaves like the number of customers in a \textsc{m/g/$\infty$} queue
with arrival rate $\lambda$ and service time distribution $p$, so that
\[
  \lim_{t\to\infty}\PP\{X(t)=k\} = e^{-r}\frac{r^k}{k!},\ \forall k\ge 0.
\]

Another point worth mentioning is that, as for the model
in~\cite{PUH1}, the Markov process $G(t)$ is reversible and hence has
an explicit invariant measure. To see this, notice that at each vertex
$v$, leaves are added at rate $\lambda$, and removed at rate
$\mu\eta(v)$, where $\eta(v)$ stands for the number of leaves attached
to $v$. Therefore, the stationary probability of some configuration
$G$ is
\[
  \pi(G)\egaldef K\frac{\rho^{N_G}}{\prod_{v\in G}\eta(v)!},
\]
where $K$ is a normalization constant. The ergodicity of the Markov
process $G(t)$ is then equivalent to the convergence of the series
$\sum\pi(G)$, where the sum is taken over all admissible trees $G$.
However, while counting Catalan trees as in~\cite{PUH1} is not that
difficult, the combinatorics is  more involved in the present setting,
and  this direction will be pursued no further.

\subsection{Volume of the tree}
Let $N\egaldef\lim_{t\to\infty}N(t)$, where $N(t)$ introduced
in section~\ref{NOTATION} stands for the volume of $G(t)$ and the
limit is taken in distribution. 

\begin{theorem}\label{thm.lawN}
When $\rho\leq e^{-1}$, the distribution of the stationary volume $N$
is given by
\begin{equation}\label{eq.N=k}
  \PP\{N=k\} = \frac{1}{r}\frac{k^{k-1}}{k!}\rho^k,
\end{equation}
where $r$ is given by (\ref{ROOT}). Moreover, the mean value of $N$
is given by
\begin{equation}\label{eq.meanN}
\EE N = \frac{1}{1-r}.
\end{equation}
\end{theorem}

\begin{proof}
We proceed as in lemma~\ref{SYS}, saying that the number of vertices
in the tree at time $t$ is equal to $1$ plus the numbers of vertices
in all the descendants that have appeared in $[0,t]$ and are not yet
dead. The construction mimics the former proposed for the process
$X_v(t)$: the volume of a subtree rooted at some vertex $v$ is
distributed as $N(t)$ for $t\leq \tau_v$.

For any complex number $z$, $|z|<1$, we have therefore
\begin{eqnarray}
\EE z^{N(t)} 
   &=& z\sum_{k=0}^\infty \frac{e^{-\lambda t}(\lambda t)^k}{k!}
       \biggl\{\int_0^t \frac{dx}{t}\EE \Bigl[z^{N(x)\ind{x\leq\tau}}
                                   \Bigr]\biggr\}^k\nonumber\\[0.2cm]
   &=& z\exp\biggl\{\lambda \EE \Bigl[\int_0^t\bigl(z^{N(x)}-1\bigr)
                                             \ind{x\leq\tau}dx\Bigr]\biggr\},
                                 \label{eq.lawNt}
\end{eqnarray}
and, letting $t\to\infty$,
\begin{equation}
  \EE z^{N} 
     = z\exp\biggl\{\lambda \EE\Bigl[\int_0^\tau\bigl(z^{N(x)}-1\bigr)
                                             dx\Bigr]\biggr\}
     = \frac{\rho z}{r}\exp\biggl\{\lambda \EE\Bigl[\int_0^\tau z^{N(x)}
                                             dx\Bigr]\biggr\}.\label{eq.lawN}
\end{equation}

It is easy to write a renewal equation similar to~(\ref{EQU2}), namely
\[
 \EE z^{N(t)} = \EE[z^{N(t)}\ind{t\leq\tau}] 
              + \int_0^t \EE z^{N(t-y)}dp(y),
\]

which, after setting $\phi(z,t)\egaldef\EE z^{N(t)}$ and
$\widetilde\phi(z,t)\egaldef\EE[z^{N(t)}\ind{t\leq\tau}]$, and taking 
 Laplace transforms with respect to $t$, yields the equation
\[
  \phi^*(z,s)=\widetilde\phi^*(z,s)+\phi^*(z,s)sp^*(s).
\]
Then, as in the case of~(\ref{NEC}), using the boundedness of
$z^{N(t)}$, we get
\[
\phi(z)
  \egaldef \lim_{t\to\infty}\phi(z,t) 
  = \frac{1}{m}\int_0^\infty\widetilde\phi(z,t)dt
  = \frac{1}{m}\EE\Bigl[\int_0^\tau z^{N(x)}dx\Bigr].
\]

Then~(\ref{eq.lawN}) can be rewritten as
\begin{equation}\label{eq.phi}
  r\phi(z) = \rho z \exp\bigl\{r\phi(z)\bigr\},
\end{equation}
and hence $r\phi(z)=C(\rho z)$, where $C$ stands for the classical
Cayley tree generating function (see e.g.~\cite{SED}). Using the
well-known series expansion for $C$ (which follows from Lagrange's
inversion formula), we get~(\ref{eq.N=k}), since
\[
 \phi(z) = \frac{1}{r}\sum_{k=0}^\infty\frac{k^{k-1}}{k!}(\rho 
 z)^k.
\]

The mean~(\ref{eq.meanN}) is obtained by
differentiating~(\ref{eq.phi}) with respect to $z$ and taking $z=1$.
\end{proof}

The analysis of the asymptotics of~(\ref{eq.N=k}) with respect to $k$
confirms an interesting change of behavior when $\rho=e^{-1}$. Indeed,
for $k$ sufficiently large, Stirling's formula yields
\[
 \PP\{N=k\} 
   = \frac{1}{r}\frac{k^{k-1}}{k!}\rho^{k} 
   \approx \frac{1}{r}\frac{k^{k-1}}{\sqrt{2\pi k}\,k^k\,e^{-k}}\rho^k
   = \frac{1}{r}\frac{(\rho e)^{k}}{\sqrt{2\pi}\,k^{\frac{3}{2}}}.
\]
Moreover, a straightforward Taylor expansion of~(\ref{ROOT}) gives the
following estimate of $\EE N$, as $\rho\to e^{-1}$:
\[
 \EE N = \frac{1}{1-r}\approx\frac{1}{\sqrt{2(1-\rho e)}}.
\]

Thus, while all moments of $N$ exist for $\rho<e^{-1}$, there is
no finite mean as soon as $\rho=e^{-1}$. We note in passing that this
phenomenon appears sometimes in branching processes and can be viewed
as a phase transition inside the parameter region, as already remarked
in \cite{FAY}.

\subsection{Height of the tree}

Let $H\egaldef\lim_{t\to\infty}H(t)$, where $H(t)$ introduced in
section~\ref{NOTATION} stands for the height of $G(t)$. The
distribution of $H$ is given by the following theorem.

\begin{theorem}\label{thm.lawH}\mbox{}
\begin{enumerate}
\item For $\rho\leq e^{-1}$, the following relations hold:
\begin{eqnarray}
\PP\{H= 0\} 
  & = &  e^{-r} , \label{eq.lawH0}\\
\PP\{H> h+1\} 
  & = & 1-\exp\Bigl[-r\PP\{H> h\}\Bigr],
        \quad \forall h \ge 0 \,.\label{eq.lawH}
\end{eqnarray}

\item If $\rho<e^{-1}$, then there exists a positive constant
$\theta(r,1)$, such that
\begin{equation}\label{eq.asymptHexp}
  \PP\{H>h\} = \theta(r,1) r^{h+1} 
         + O\Bigl(\frac{r^{2h}}{1-r}\Bigr)\,
\end{equation}
where the function $\theta(r,x)$ is the locally analytic w.r.t.\ $x$
solution of the functional Schr\"oder equation
\[
\theta(r,1-e^{-rx}) = r\theta(r,x),
\]
subject to the boundary condition
\begin{equation}\label{eq:condschroder}
 \frac{\partial\theta}{\partial x}(r,0)=1.
\end{equation}
\item When
$\rho=e^{-1}$,
\begin{equation}\label{eq.asymptHexp1}
 \PP\{H>h\} = \frac{2}{h}+O\Bigl(\frac{\log h}{h^2}\Bigr)\,.
\end{equation}
\end{enumerate}
\end{theorem}

\begin{proof}
As in the previous proof, one writes the height of the tree at time
$t$ is less than $h+1$ if, and only if, all the descendants that have
appeared in $[0,t]$ are either dead or have a height smaller than $h$:
\begin{eqnarray*}
\PP\bigl\{H(t)\leq h+1\bigr\} 
   &=& \sum_{k=0}^\infty \frac{e^{-\lambda t}(\lambda t)^k}{k!}
       \Bigl\{\int_0^t \frac{dx}{t}
              \bigl[1- \PP\bigl\{H(x)>h, x\leq\tau\bigr\}\bigr]\Bigr\}^k\\
   &=& \exp\Bigl\{-\lambda\int_0^t\PP\bigl\{H(x)>h,x\leq\tau\bigr\}dx
                                                                   \Bigr\}.
\end{eqnarray*}

Letting $t\to\infty$ and arguing as in theorem~\ref{thm.lawN}, we can
write
\[
\PP\bigl\{H\leq h+1\bigr\} 
   = \exp\Bigl[-r \PP\bigl\{H>h\bigr\}\Bigr],
\]
which proves~(\ref{eq.lawH}). On the other hand, equation
(\ref{eq.lawH0}) is immediate, since it is in fact a plain rewriting
of (\ref{PX0}).

To prove the remainder of the theorem, let $d_0\egaldef x$, where $x$ is a
positive real number, and consider the sequence
\begin{equation}\label{RECdh}
d_{h+1}= 1 - e^{-rd_h}, \ h=0,1,...
\end{equation}

When $x=1$, note that we have exactly $d_{h+1}=\PP\bigl\{H>h\bigr\}$.
The question that faces us now is to compute and to estimate the
iterates of an analytic function, in the circumstances $1-e^{-rx}$.
This subject concerns a wide branch of mathematics (including
functional equations, automorphic functions, boundary value problems),
and it has received considerable attention since the nineteen
twenties. We shall employ classical arguments without further comment,
referring the interested reader to e.g.~\cite{KUC} and
\cite{deB} for a more extensive treatment.

A Taylor expansion up to second order in (\ref{RECdh}) gives
\begin{equation}\label{TAYdh}
rd_h \ge d_{h+1} = 1 - e^{-r d_h} \ge r d_h -\frac{r^2d_h^2}{2},
\end{equation}
which implies that $r^{-h} d_h$ is a decreasing sequence with
$\DD\lim_{h\to\infty} \downarrow d_h=0$ (that we already knew!)  and
\begin{equation}\label{GEOM}
d_h\le xr^{h}. 
\end{equation}

As $h\to\infty$, the asymptotic behavior of $d_{h}$ has a twofold
nature, depending on whether $r=1$ or $r<1$.

\paragraph{Case $r=1$.}
This is the easy part. Writing
\[
\frac{1}{d_{h+1}}= \dfrac{1}{1 - e^{-d_h}}= \frac{1}{d_h}+\frac{1}{2}
+O(d_h),
\]
we get immediately $d_h= O\Bigl(\frac{1}{h}\Bigr)$, and hence
\[
\frac{1}{d_h}= \frac{h}{2} +O\Bigl(\sum_{k=0}^{h-1}d_k\Bigr) 
             = \frac{h}{2} +O(\log h), 
\]
which leads  to (\ref{eq.asymptHexp1}).

\paragraph{Case $r<1$.}  
The analysis is less direct. From (\ref{TAYdh}) and (\ref{GEOM}), we
infer that, when $h\to\infty$, $r^{-h}d_h$ has a limit denoted by
$\theta(r,x)$, with
\[
0\le r^{-h}d_h - \theta(r,x) \le\frac{rx^2}{2}\frac{r^h}{1-r}.
\]

First let us show that $\theta(r,x)$ is strictly positive.  Indeed,
\begin{equation}\label{PROD}
\frac{d_{h+1}}{r^{h+1}} = x\prod_{m=0}^h \bigl(1 -\varphi_{m}(r,x)\bigr),  
\end{equation}
where, $\forall m\ge 0$, the quantity $\varphi_{m}(r,x) = O(r^{m+1})$
is an analytic function of the pair of real variables $(r,x)$ in the
region $[0,1[\times[0,A]$, with $0\le A<\infty$. Hence, as
$h\to\infty$, the infinite product in (\ref{PROD}) converges uniformly
to a strictly positive value, $\forall x>0$, so that $\theta(r,x)$ is
also  analytic of $(r,x)$ in the aforementioned region. To summarize,
\[
\lim_{h\to\infty} r^{-h} d_h  \egaldef \theta(r,x) > 0.
\]
The pleasant fact is that $\theta$, taken as a function of $x$,
satisfies the so-called Schr\"oder equation
\begin{equation}\label{SCH}
\theta(r,1-e^{-rx}) = r\theta(r,x).
\end{equation}

While it is clear that $\theta(r,0)=0$, (\ref{SCH}) does not impose
any constraint on $\DD\frac{\partial\theta}{\partial x}(r,0)$.
However, it is easy to show by induction that $d_h(x)$ (where the dependency
on $x$ is for a while explicitly written) satisfies
\[
  \DD\frac{\partial d_h}{\partial x}_{|x=0}= r^h, \ \forall h>0,
\]
and thus condition~(\ref{eq:condschroder}) also holds for $\theta$. To
conclude the proof of~(\ref{eq.asymptHexp}), it suffices to choose
$x=1$.
\end{proof}

\paragraph{Remark} We have taken the variable $x$ on the positive real 
half-line to get sharper bounds, e.g.~(\ref{GEOM}). Actually, arguing
as above, it is immediate to check that $\theta$ has an analytic
continuation in the complex $x$-plane in a a neighborhood of the
origin. In this respect, without going into a full discussion, we
mention the relationships with automorphic functions and boundary
value problems, which would allow integral representations. For our
purpose, simply writing
\[
  \theta(r,x) =\sum_{i\ge 0} \theta_{i}x^i, \quad \theta_0=0, \  \theta_1= 1 ,
\]  
we see that all the $\theta_i$'s can be computed recursively. Furthermore
the iteration of (\ref{SCH}) yields
\[
  \theta(r,d_{h}) = r^h \theta(r,x).
\]
from which we obtain 
\[
  d_{h}= \omega \bigl(r,r^h\theta(r,x)\bigr),
\]
where $\omega (r,x)$ denotes the inverse function of $\theta$ with
respect to the variable~$x$ and satisfies  the functional relation
\begin{equation}\label{INV}
1-\exp \{-r\omega(r,y)\} = \omega(r,ry).
\end{equation}
We have 
\[
\omega(r,y) =\sum_{{i\ge 0}} \omega_{i}y^i , \quad \omega_0=0, \ \omega_1=1,
\]
and again the $\omega_{i}$'s are obtained recursively. 

\section{Scaling and limit laws in the transient case}\label{GROWTH}

In this section, we present some limit laws for
$N(t)$ and $H(t)$, which are especially of interest when the system is
transient.  Beforehand, for every integer $k$ and all $t>0$, we define
the quantities
\[
\begin{cases}
\DD X_k(t) \ \egaldef \ \#\bigl\{v \in G(t): h(v) = k \bigr\},\\[0.3cm]
\DD Y_k(t) \ \egaldef \ \sum_{j=k}^{\infty} X_j(t)\ind{t\leq\tau}.
\end{cases}
\]
So, $X_k(t)$ stands for the number of vertices at  level~$k$ in the
whole tree at time~$t$.

\subsection{Scaling for $N(t)$ in the pure birth case $\mu=0$} 
\begin{lemma}\label{lem.meanXn}
When $\mu=0$, $\EE X_n(t)$ has the explicit form
\begin{equation}\label{Xn}
\EE X_n(t)=\frac{(\lambda t)^n}{n!}.
\end{equation}
\end{lemma}

\begin{proof} Since
\[\PP\bigl\{X_n(t+dt)=X_n(t)+1\big|X_n(t),X_{n-1}(t)\bigr\}
   = \lambda X_{n-1}(t)dt+o(dt),\]
we obtain 
\[
\begin{cases}
\dfrac{d}{dt}\EE X_n(t)=\lambda \EE X_{n-1}(t), \quad n\geq 1, \\[0.3cm]
 \EE X_0(t)=1, 
\end{cases}
\]
and the result is immediate by induction.
\end{proof}

\begin{theorem}\label{thm.birthN}
When $\mu=0$, the expected volume at time $t$ is given by
\begin{equation}\label{eq.birthN.mean}
 \EE N(t) = e^{\lambda t},
\end{equation}
and 
\begin{equation}\label{eq.birthN.law}
\lim_{t\to\infty}\frac{N(t)}{\EE N(t)} = \mathrm{Exp}(1) \,,
\end{equation}
where the limit is taken in distribution and $\mathrm{Exp}(1)$ denotes
an exponentially distributed variable with parameter $1$.
\end{theorem}

\begin{proof}
Equation~(\ref{eq.birthN.mean}) is a mere consequence of
lemma~\ref{lem.meanXn}, since
\[
  \EE N(t) = \sum_{n=0}^{\infty}\EE X_n(t) = e^{\lambda t}.
\]

Let now $\phi(z,t)\egaldef\EE z^{N(t)}$, for $z$ complex with
$|z|<1$. We start from equation~(\ref{eq.lawNt}), in which we take
$\tau=\infty$. Then the following relation holds:
\[
 \phi(z,t) =  z\exp\biggl\{\lambda \Bigl[\int_0^t\bigl(\phi(z,x)-1\bigr)
                                             dx\Bigr]\biggr\}.
\]
Differentiating with respect to $t$ yields
\[
 \frac{\partial}{\partial t}\phi(z,t) 
     = \lambda\bigl[\phi(z,t)-1\bigr]\phi(z,t),
\]
whence 
\[
  \frac{1-\phi(z,t)}{\phi(z,t)} = K e^{\lambda t},
\]
where $K$ does not depend on $t$. Since $\phi(z,0)=z$, we deduce $K =
z^{-1}-1$, and finally
\[
 \phi(z,t) = \frac{1}{\DD 1+[z^{-1}-1]e^{\lambda t}}.
\]
The Laplace transform of $e^{-\lambda t}N(t)$ is, for $\Re (s)\ge 0$,
\begin{eqnarray*}
\EE \exp\bigl\{-s e^{-\lambda t}N(t)\bigr\}
  &=& \phi\Bigl(\exp\bigl\{-se^{-\lambda t}\bigr\},t\Bigr)\\
  &=& \frac{1}{\DD 1 + 
               \bigl[\exp\{se^{-\lambda t}\}-1\bigr]e^{\lambda t}},
\end{eqnarray*}
so that, letting $t\to\infty$, 
\[
 \lim_{t\to\infty} \EE \exp\bigl\{-s e^{-\lambda t}N(t)\bigr\}
   = \frac{1}{1+s}.
\]
Now (\ref{eq.birthN.law}) follows directly from Feller's continuity
theorem (see~\cite{FEL}).
\end{proof}

\subsection{An ergodic theorem for $H(t)$}
The key result of this section concerns the height of the tree
and is formulated in the next theorem.

Let 
\begin{equation}\label{BSC} 
b(s,c)\egaldef \frac{s}{c} + \log
 \left[\frac{\lambda(1-sp^*(s))}{s}\right], \quad \Re(s)\ge 0.
\end{equation}  
\begin{theorem}\label{thm.transientH}
With probability $1$, 
\[
 \lim_{t\to\infty} \frac{H(t)}{t} = \delta,
\]
where $\delta \ge 0$ is uniquely defined from  the system of equations
\[
 b(s,\delta)=\frac{\partial b(s,\delta)}{\partial s}=0.
\]
 In the ergodic case, $\delta=0$.
\end{theorem}
The proof is constructed around the three forthcoming lemmas.

\begin{lemma}\label{LEMMA-AB}
Define the events
\[
A_c=\Bigl\{\liminf_{t\to\infty}\frac{H(t)}{t}\ge c\Bigr\}, \qquad
B_c=\Bigl\{\limsup_{t\to\infty}\frac{H(t)}{t}\le c\Bigr\}.
\]
Then $\PP\{A_c\}=0$ or $1$ and $\PP\{B_c\}=0$ or $1$. In other
words, $A_c$ and $B_c$ satisfy a zero-one law and can only be trivial
events (i.e. sure or impossible).
\end{lemma}

\begin{proof}
Fixing an arbitrary $t_0$, with $G(t_0)=G_0$, we want to show that
$A_c$ does not depend on $G_0$. For this purpose, consider the random
process $G'(t)\subset G(t)$ constructed as follows: for $t\le t_0$, it
consists only of the root, and for $t>t_0$ it contains exactly that
part of $G$ grown from the root after time $t_0$. Then the probability
that $A_c$ holds for $G'$ is clearly equal to the probability that
$A_c$ holds for $G$ without conditioning. In other words, since
$H_G(t)\ge H_{G'}(t)$, we have
\[
 \PP\{A_c\mid G(t_0)=G_0\} \ge \PP\{A_c\}.
\] 
Then basic properties of the conditional expectation yield
\[
\EE [\PP\{A_c\mid G(t_0)\}] = \PP\{A_c\},
\]
so that
\begin{equation}\label{EventA}
\PP\{A_c\mid G(t_0)=G_0\} = \PP\{A_c\}
\end{equation}
for any $G_0$. On conditioning with respect to $G(t_0), G(t_1),\ldots,
G(t_k)$, for any arbitrary increasing sequence of times $t_k$, we see
(\ref{EventA}) still holds. Hence, the assertion for $A_c$ is a direct
consequence of the zero-one law for martingales (see~e.g.~\cite{KAL}).

Quite similarly, if the event $B_c$ holds for $G$, then it is also in
force for any subtree rooted at a vertex of $G_0$, which reads
\[
\PP\{B_c\mid G(t_0)=G_0\} \le \PP\{B_c\}.
\]
The lemma is proved.
\end{proof}

\begin{lemma}\label{ZERO-ONE} \mbox{ }
\begin{enumerate} 
\item[(i)] If, for some integer $n$ and real number $c>0$,
$\EE[Y_n(n/c)]>1$, then
\[
  \PP\{A_{c}\}=1.
\]
\item[(ii)] If, for some $n$ and real number $c>0$, 
$\DD \sum_{k=0}^\infty \EE[X_{kn}(kn/c)]<\infty$,
then
\[
  \PP\{B_{c}\}=1.
\]
\end{enumerate}
\end{lemma}
\begin{proof} \mbox{}\\
For the sake of brevity, let $J_n(k)$ denote the time interval 
$[kn/c, (k+1)n/c]$.
 
\emph{(i)} Consider a standard branching process ${\xi_k,k\ge0}$,
endowed with an offspring distribution equal to that of $Y_n(n/c)$.
From the condition in 4.3(i), this process has a probability of non
extinction which is strictly positive and will be denoted by $y(n,c)$.
The key point is that ${\xi_k}$ can be viewed as defining a subtree
$G'\subset G$ such that $H_{G'}(kn/c)\ge kn$. Then $\PP\{A_{c}\}\ge
y(n,c)>0$, and we have $\PP\{A_{c}\}=1$ by lemma~\ref{LEMMA-AB}.

Indeed, to build such a subtree $G'$, we associate with each generation of
$\xi_k$ a set of vertices $s_k\subset G(k\tau)$, such that
$\xi_k=|s_k|$. 

Let $s_0=\{v_0\}$ and, for each $v\in s_k,k\ge0$, let $G_v$ be a
subtree rooted at $v$ and born during $J_n(k)$ (by convention $G_v$ is
empty if $v$ dies). We put
\[
s_{k+1}=\bigcup_{v\in s_k} \{v'\in G_v: d(v',v)\ge n\}.
\]
This construction produces the desired tree, since the volume of each
set belonging to the above union is exactly distributed as $Y_n(n/c)$,
and because $s_k$ consists of vertices located at a distance at least
$kn$ from the root.

\emph{(ii)} 
Let $a_k=o(k)$ be a sequence of non-decreasing positive integers. Then
for any fixed integer $k_0$, we have the inequality
\[
\PP(B_c) \ge 
\PP\Bigl\{\sup_{t\in J_n(k)} H(t)<(k+1)n+a_k, \forall k \ge k_0 \Bigr\}.
\]
We observe the height of the tree  decreases at a rate not faster than 
$\mu$, so that, given the event $\{H((k+1)n/c)<(k+1)n\}$, the supremum of 
$H(.)$  on the interval $J_n(k)$ is bounded by 
\[
 \sup_{t\in J_n(k)} H(t) \le (k+1)n + \pi(\mu n/c),
\]
where $\pi(x)$ denotes a Poisson random variable with rate $x$. Thus
we have
\begin{equation}\label{POS}
\PP(B_c) \ge \PP\{H(kn/c) < kn, \forall k \ge k_0 \} 
\prod_{k\ge k_0} \PP\{\pi(\mu n/c) < a_k \},
\end{equation}
and we will show that the right-hand side of (\ref{POS}) can be
 rendered positive. First, we remark that
\begin{eqnarray*}
\PP\{H(kn/c)<kn, \forall k \ge k_0 \} 
&\ge & 1-\sum_{k\ge k_0} \PP\{H(kn/c)\ge kn\} \\[0.2cm] &\ge &
1-\sum_{k\ge k_0} \EE[X_{kn}(kn/c)]\to 1, \ \mathrm{as}\
{k_0\to\infty}.
\end{eqnarray*}
Secondly, we choose the sequence 
\[a_k=j, \ \forall k\in[j(j-1)/2 +1,\,
j(j+1)/2], \forall j\ge 1,
\]
which consists of blocks of repeated integers satisfying  
$a_k=\mathcal{O}\bigl(k^{1/2}\bigr)$.

Setting $\nu\egaldef\mu n/c$, the product in
(\ref{POS}) will be positive, provided that the following sum is finite
\begin{eqnarray*}
\sum_{k\ge k_0} \PP\{\pi(\nu)\ge a_k \} &\le &
\sum_{k\ge k_0} e^{-\nu} \frac{\nu ^{a_k}}{a_k!} 
\Big(1-\frac{\nu}{a_k}\Big)^{-1} \\
&\le &\sum_{j\ge \sqrt{k_0}}  e^{-\nu}
\frac{\nu^j}{(j-1)!}\Big(1-\frac{\nu}{j}\Big)^{-1} \le \nu
\Big(1-\frac{\nu}{\sqrt{k_0}}\Big)^{-1} < \infty ,
\end{eqnarray*}
and hence \emph{(ii)} follows from the zero-one property of $B_c$.

The proof of the lemma is concluded.
\end{proof}

\begin{lemma}\label{lem.lapl}
For any $\Re(s)>0$, let $\varphi_k(s)$ and $\wt\varphi_k(s)$ be the
Laplace transforms
\begin{eqnarray*}
\varphi_k(s) 
  &\egaldef& \int_0^\infty \EE X_k(t)\,e^{-st}dt,\\
\wt\varphi_k(s) 
  &\egaldef& \int_0^\infty \EE \bigl[X_k(t)\ind{t\leq\tau}\bigr]\,e^{-st}dt.
\end{eqnarray*}
Then
\begin{eqnarray*}
\varphi_k(s) &=&  \frac{\lambda^k (1-sp^*(s))^k}{s^{k+1}} , \\
\wt\varphi_k(s) &=& \frac{\lambda^k (1-sp^*(s))^{k+1}}{s^{k+1}}.
\end{eqnarray*}
\end{lemma}

\begin{proof}
It is not difficult to check the following relations
\begin{eqnarray*}
\EE X_k(t) 
  &=& \int_0^t \EE \bigl[X_{k-1}(y)\ind{y\leq\tau}\bigr] \lambda dy, 
      \quad k\ge 1 \\
\EE X_k(t) 
  &=& \EE \bigl[X_k(t)\ind{t\leq\tau}\bigr] + \int_0^t 
  \EE X_k(t-y) dp(y), 
\end{eqnarray*}
with the initial condition $\EE X_0(t)=1$. Actually, the first
equation follows from an argument already employed before. Namely, the
number of vertices at level $k$ are the direct descendants of vertices
at level $(k-1)$ still alive at time $t$, remarking that each such
descendant on $[0,t]$ appears independently at rate $\lambda$. The
second equation is a straight regeneration relation. Therefore,
\begin{eqnarray*}
\varphi_k(s) &=& \frac{\lambda\wt\varphi_{k-1}(s)}{s} \quad k\ge 1, \\[0.2cm]
\varphi_k(s) &=& \wt\varphi_k(s) + \varphi_k(s)sp^*(s),
\end{eqnarray*}
whence, since $\varphi_0(s)=1/s$,
\[
\varphi_k(s) = \frac{\lambda \bigl(1-sp^*(s)\bigr)\varphi_{k-1}(s)}{s}=
 \frac{\lambda^k\bigl(1-sp^*(s)\bigr)^k}{s^{k+1}} , 
\]
and the result follows.
\end{proof}

We are now in a position to prove theorem~\ref{thm.transientH}.
\begin{proof}[\ifspringer\else Proof \fi of theorem~\ref{thm.transientH}]
The proof is split into two parts, each one corresponding
respectively to criteria \emph{(i)} and \emph{(ii)} of
lemma~\ref{ZERO-ONE}.
 
  First we shall find the \emph{largest} $c$, denoted by
  $c_{\text{inf}}$, ensuring criterion \emph{(i)} of
  lemma \ref{ZERO-ONE} is fulfilled. Applying the results of
  lemma~\ref{lem.lapl} and the inversion formula~(\ref{LAP}), we have
\begin{eqnarray}
\EE[Y_n(n/c)]
  &=& \sum_{j=n}^\infty \frac{1}{2i\pi}\int_{\sigma-i\infty}^{\sigma+i\infty}
\wt\varphi_j(s)e^{sn/c}ds \nonumber \\[0.2cm] 
&=&\frac{1}{2i\pi}\int_{\sigma-i\infty}^{\sigma+i\infty}
\left[\frac{\lambda(1-sp^*(s))}{s}e^{s/c}\right]^n
\frac{1-sp^*(s)}{s-\lambda(1-sp^*(s))}ds,\nonumber \\[0.2cm]
 &=& \frac{1}{2i\pi}\int_{\sigma-i\infty}^{\sigma+i\infty}
\frac{e^{nb(s,c)}\,[1-sp^*(s)]} {s-\lambda(1-sp^*(s))}ds, 
\qquad\label{INTYn} 
\end{eqnarray} 
in the region $\DD\U\egaldef \bigl\{\sigma>0,\ \sigma>\lambda
(1-\sigma p^*(\sigma))\bigr\}$, where $b(s,c)$ has been defined in
(\ref{BSC}). When the system is ergodic, it is immediate to check the
region $\U$ coincides with the complex half-plane $\Re (s)>0$. On the
other hand, in the transient case, the equation
\[
s=  \lambda(1-sp^*(s))
\] 
has exactly one root, which is real and belongs to the open interval
$]0,\lambda[$. Computing the residues of the integral (\ref{INTYn})
(by shifting the line of integration to the left, after analytic
continuation of $p^*(s)$ to the region $\sigma =-\epsilon$, for some
$\epsilon>0$) is a tedious task, in particular due to the pole of order
$n$ at $s=0$. We will rather proceed by a kind of \emph{saddle-point}
approach (see e.g.~\cite{FUC}).
 
The form of the integrand in (\ref{INTYn}) shows that, as
$n\to\infty$, the boundedness of $\EE[Y_n(n/c)]$, depends primarily on
the value of the modulus of $b(s,c)$. In fact one can see precisely
that $\EE[Y_n(n/c)]$, for each fixed $c$, does not tend to zero iff
the minimum of $b(s,c)$ is non-negative at any possible real
saddle-point $s\in\U$, where
\[
\frac{\partial b(s,c)}{\partial s}=0, \quad s \in\U.
\]
It follows that $c_{\text{inf}}$ is the unique real solution of the
system
\begin{equation}\label{CC}
b(s,c_{\text{inf}})= \frac{\partial b(s,c_{\text{inf}})}{\partial
s}=0, \quad s\in\U.
\end{equation}
Without presenting a detailed discussion, we shall simply stress that
in the real plane $(s,y)$ the curves
\[
y=s/c \quad \text{and}\quad y= -\log
\left[\frac{\lambda(1-sp^*(s))}{s}\right]
\]
 are tangent (resp.\ intersecting, non-intersecting) for
$c=c_{\text{inf}}$ (resp. $c>c_{\text{inf}}$,
$c<c_{\text{inf}}$).

As for the second part of the theorem, the question is to find the
value $c_{\text{sup}}$, equal to the smallest positive number $c$
satisfying criterion \emph{(ii)} of lemma~\ref{ZERO-ONE}, which
implies the finiteness of the quantity
\begin{eqnarray}
\sum_{k=0}^\infty \EE[X_{kn}(kn/c)] &=& \sum_{k=0}^\infty
\frac{1}{2i\pi}\int_{\sigma-i\infty}^{\sigma+i\infty}
\varphi_{kn}(s)e^{skn/c}ds \nonumber \\[0.2cm]
 &=&
\frac{1}{2i\pi}\int_{\sigma-i\infty}^{\sigma+i\infty}
\frac{ds}{s}\left[1-\left(\frac{\lambda(1-sp^*(s))}{s}e^{s/c}\right)^n
\right]^{-1} \nonumber \\[0.2cm] 
 &=&
\frac{1}{2i\pi}\int_{\sigma-i\infty}^{\sigma+i\infty}
\frac{ds}{s}\bigl[1-e^{nb(s,c)}\bigr]^{-1}, \qquad\label{INTXn} 
\end{eqnarray}
where (\ref{INTXn}) holds in the region $\V\egaldef
\bigl\{\sigma>0,\ \sigma> \lambda(1-\sigma p^*(\sigma)) e^{\sigma/c}\bigr\}$.
 
Clearly, the existence of the last integral in (\ref{INTXn}), as
 $n\to\infty$, amounts again to find the sign of $\Re (b(s,c))$, for
 $s\in\V$. Arguing exactly as above, one can find at once the equality
\[
c_{\text{sup}}=c_{\text{inf}}\egaldef\delta.
\]
When the system is ergodic, $\lim_{s\to 0} b(s,c)= \log \lambda m =
 \log r \le 0$, which yields $\delta=0$ as might be expected.

The proof of the theorem is concluded. \end{proof} As a by-product, we
 state the following corollary, of which the almost sure convergence
 part has been derived in~\cite{DEV,PIT} through different and less
 terse methods.

\begin{corollary}\label{MU0}
 In the pure birth case $\mu=0$, almost surely and in $L_1$, 
 \begin{equation}\label{eq:MU0}
 \lim_{t\to\infty} \frac{H(t)}{t} = \lambda e . 
 \end{equation}
 \end{corollary}

\begin{proof}
Instantiating equation (\ref{Xn}) in criteria \emph{(i)} and
 \emph{(ii)} of lemma~\ref{ZERO-ONE} yields directly the first part of
 (\ref{eq:MU0}). On the other hand, it is immediate to check that the
 function $\EE H(t)$ is superadditive (this would be not true for
 $\mu\ne0$), namely
\[
\EE H(s+t) \geq \EE H(s) + \EE H(t),
\]
so that, by a variant of Kingman's theorem (see~\cite{KAL}), the limit
 $\DD\lim_{t\to\infty}\frac{\EE H(t)}{t}$ does exist. Then the
 convergence in $L_{1}$ will follow if one can show
\[\EE H(t)\le At , \quad \forall t>0,
\]
for some positive constant $A$. Using the obvious inequality
\[\PP\{H(t)\ge
k\}=\PP\{X_k(t)>0\}\le \EE X_k(t),
\] we can write
\[
\EE H(t)=\sum_{k=1}^\infty \PP\{H(t)\ge k\}
\le \sum_{k=1}^{k_0} 1 + \sum_{k=k_0+1}^\infty \frac{(\lambda t)^k}{k!}.
\]
Then, taking  $k_0=\lceil \lambda et \rceil$ and using Stirling's
formula, we obtain
\begin{eqnarray*}
\EE H(t) 
  &\le &  k_0+\frac{(\lambda t)^{k_0+1}}{(k_0+1)!}
                   \sum_{k=0}^\infty\Bigl(\frac{\lambda t}{k_0+1}\Bigr)^k
\le \lambda et+\frac{(\lambda t)^{\lambda et+1}}
                    {(\lambda t)^{\lambda et}\sqrt{2\pi \lambda et}}
               \frac{e}{e-1}\\
&\le& \lambda et + \frac{\sqrt{\lambda et}}{\sqrt{2\pi}(e-1)}.
\end{eqnarray*}
\end{proof}

\section{Extension to the multiclass case}\label{MULTI}

\newcommand{\C}{\mathcal{C}}

The extension of the results of section~\ref{DELETE} to models
encompassing several classes of vertices is very tempting, although
not quite evident. We solve hereafter a case where the birth and death
parameters depend on classes in a reasonably general way.

Let $\C$ be a finite set of classes. Then the multiclass Markov chain
$G_\C$ has the following evolution.
\begin{itemize}
\item At any given vertex of class $c$, a new edge of class $c'\in\C$
can be added at the epochs of a Poisson process with parameter
$\lambda_{cc'}\geq0$.
\item Any leaf attached to an edge of class $c'$ and having an
ancestor of class $c$ can be deleted at rate $\mu_{cc'}>0$.
\item The root $v_0$ of the tree belongs to class $c\in\C$, say with
probability $\pi_c$, with $\sum_{c\in\C}\pi_c=1$, albeit these
probabilities will not really matter in our analysis.
\end{itemize}

Let $p_{cc'}$ be the lifetime distribution of a vertex of class $c'$
which descend from a vertex of class $c$. Also, $X_c(t)$ will denote
the distribution of the number of direct descendants of a vertex of
class $c$. The following lemma is the analogous of lemma~\ref{SYS}.

\begin{lemma}
The lifetime distributions $p_{cc'}$, $c,c'\in\C$ satisfy the
following set of equations.
\begin{eqnarray}
\PP\{X_c(t)=0\} 
  &=& \exp\Bigl\{-\sum_{c'\in\C}\lambda_{cc'}
                                     \int_0^t(1-p_{cc'}(x))dx\Bigr\}
      \label{eq.alphaexp},\\
\PP\{X_{c}(t)=0\} 
  &=& \frac{1}{\mu_{bc}}\frac{dp_{bc}(t)}{dt}
      +\int_0^t\PP\{X_c(t-y)=0\} dp_{bc}(y),\ \forall b\in\C, \nonumber
         \\  \label{eq.alphareg}
\end{eqnarray}
with the initial conditions $p_{cc'}(0)=0, \forall c, c'\in\C$.
\end{lemma}
\begin{proof}
Details are omitted, as it suffices to mimic the proof of
lemma~\ref{SYS}. Note however that the dependency with respect to $c'$
disappears surprisingly enough in (\ref{eq.alphareg}). Indeed, while
the lifetime of a vertex depends on the class of its direct ascendant,
the distribution of the number of its descendants merely depends on
its own class.
\end{proof}

In the setting of this section, it is actually not easy to come up
with a natural explicit extension of theorem~\ref{TH2}. However, the
following theorem provides a necessary and sufficient condition for
ergodicity.

The following notation will be useful in the theorem:
\begin{eqnarray*}
  \rho_{cc'}&\egaldef&\frac{\lambda_{cc'}}{\mu_{cc'}},\ \forall c,
  c'\in\C\\
  \rho_c&\egaldef&\sum_{c'\in\C}\rho_{cc'},\ \forall c\in\C.
\end{eqnarray*}

We will also denote by $\rho\ge0$ the Perron-Frobenius eigenvalue
(see~\cite{GAN}) of the non-negative matrix
$\bigl(\rho_{cc'}\bigr)_{c,c'\in\C}$.

\begin{theorem}\mbox{}\label{th-MULTI}
\begin{enumerate} 
\item\label{item.ergomulti} The Markov chain $G_\C$ is ergodic if, and
  only if, the nonlinear system
\begin{equation}\label{eq.rc}
  y_c =\sum_{d\in\C}\rho_{cd} \,\exp\{y_d\},\ \forall c\in\C,
\end{equation}
has at least one real-valued (and obviously non-negative) solution. In
this case, the mean lifetime of a vertex of class $c$ with an ascendant
of class $b$ can be written as
\begin{equation}\label{eq.mc}
  m_{bc}  = \frac{1}{\mu_{bc}}\exp\{r_c\}\,,
\end{equation}
where the $r_c$ form the smallest solution of~(\ref{eq.rc}), that is
$r_c\leq y_c$, $\forall c\in\C$. Note that (\ref{eq.rc}) implies that
$r_c\geq \rho_c$.
\item\label{item.ergomulti.suf} A simple sufficient condition for
  ergodicity is
\begin{equation}\label{eq.ergomulti.suf}
  \rho_c\leq\frac{1}{e},\ \forall c\in\C
\end{equation}
in which case $r_c\leq \rho_c e$.
\item\label{item.ergomulti.nec} A simple necessary condition for
  ergodicity is
\begin{equation}\label{eq.ergomulti.nec}
  \rho \leq\frac{1}{e}.
\end{equation} 
\end{enumerate}
\end{theorem}

\paragraph{Remark} Before stating the proof of the theorem, it is
worth pointing out that equation~(\ref{eq.rc}) may in general have
several real solutions (as in dimension $1$). Therefore, there is no
guarantee that the solution $y_c$ is the correct one. However, its
sole existence proves ergodicity and (\ref{eq.mc}).

\begin{proof}
Assume first that $G_\C$ is ergodic. Then, as in
theorem~\ref{TH2}, we let $t\to\infty$ in
(\ref{eq.alphaexp})--(\ref{eq.alphareg}) to obtain the relation
\[
  \frac{1}{\mu_{bc} m_{bc}} 
         = \exp\Bigl\{-\sum_{c'\in\C}\lambda_{cc'}m_{cc'}\Bigr\},
\]
which in its turn yields (\ref{eq.rc}) and (\ref{eq.mc}), just choosing
\[
  y_c = r_c \egaldef\sum_{c'\in\C}\lambda_{cc'}m_{cc'}.
\]

As for the proof of sufficiency in item~\ref{item.ergomulti}, we
introduce the following modified version of scheme~(\ref{ITE}):
\begin{equation}\label{eq.itemulti}
\begin{cases}
 p_{cc';0}(t) & = 1,\ t\ge0 \,, \\[0.2cm] 
 \alpha_{c;k+1} (t) & = \
 \DD\exp\Bigl\{-\sum_{c'\in\C}\lambda_{cc'}
 \int_0^t\bigl(1-p_{cc';k}(y)\bigr)dy\Bigr\},\ k\ge 0, \\[0.4cm]
 \alpha_{c;k}(t) & = \ \DD\frac{1}{\mu_{bc}}\frac{dp_{bc;k}(t)}{dt} 
                     + \int_0^t\alpha_{c;k}(t-y)dp_{bc;k}(y),\ k\ge1,\\[0.3cm] 
 p_{cc';k}(0) & = \ 0, \  k \ge 1\,.
\end{cases}
\end{equation}
Then, for any $b,c\in\C$, we have
\[
\alpha_{c;1}(t) = 1 \quad\mbox{and}\quad
p_{bc;1}(t) = 1-e^{-\mu_{bc}t}\ \leq\ p_{bc;0}(t).
\]
Here again, the positive sequences $\{p_{cc';k}(t);k\geq 0\}$ and
$\{\alpha_{c;k}(t); k\geq 0\}$ are uniformly bounded and
non-increasing, for each fixed $t>0$. Consequently,
\[
p_{cc'}(t) = \limdec_{k\to\infty} p_{cc';k}(t) 
\quad \textrm{and} \quad
\alpha_c(t) = \limdec_{k\to\infty} \alpha_{c;k}(t)
\]
form the unique solution of (\ref{eq.alphaexp})--(\ref{eq.alphareg}),
uniqueness resulting from the Lipschitz character of equation
(\ref{eq.alphareg}).

Letting $m_{cc';k}$ denote the finite mean associated with each distribution
$p_{cc';k}$ and 
\[
 r_{c;k} \egaldef \sum_{c'\in\C} \lambda_{cc'}m_{cc';k} \,,
\]
we can write the following recursive equation
\[
  r_{c;k+1}=\sum_{c'\in\C}\rho_{cc'}\exp\{r_{c';k}\},\ \forall c\in\C.
\]

The $p_{cc';k}$'s are decreasing sequences, and hence the $r_{c;k}$'s
are non-decreasing, with $r_{c;0}=0, \forall c\in \C$. If
(\ref{eq.rc}) has a solution, then the relation
\[
  r_{c;k+1} - y_c \,=\,
     \sum_{c'\in\C}\rho_{cc'}\bigl[\exp\{r_{c';k}\}-\exp\{y_{c'}\}\bigr]
\]
yields $r_{c;k}\leq y_c$, for all $c\in\C$. Therefore, each
sequence $r_{c;k}$ converges as $k\to\infty$ to a finite value
$r_c\leq y_c$, and $G_\C$ is ergodic since, by (\ref{eq.mc}), the
$m_{cc'}$'s are also finite. When (\ref{eq.ergomulti.suf}) holds, the same
line of argument shows that the sequences $r_{c;k}$ are non-decreasing
and bounded from above by $\rho_c e$.

Finally, to prove~(\ref{eq.ergomulti.nec}), we use the following
inequality (see~\cite{GAN}), valid for any $x_c>0$, $c\in\C$:
\[
  \rho \leq \max_{c\in\C} \sum_{c'\in\C} \frac{\rho_{cc'}x_{c'}}{x_c}.
\]

When the $r_c$'s satisfy (\ref{eq.rc}), the choice $x_c=\exp\{r_c\}$
implies
\[
 \rho \leq \max_{c\in\C} \biggl[\sum_{c'\in\C} 
                  \rho_{cc'}\exp\{r_{c'}\}\exp\{-r_c\}\biggr]
      = \max_{c\in\C} \Bigl[r_c \exp\{-r_c\}\Bigr] \leq \frac{1}{e},
\]
which concludes the proof of the theorem.
\end{proof}

It is possible to extend the results of section~\ref{DISTRIBUTIONS} to
the multiclass case. We will only sketch the proofs in what follows,
since they are very similar to their single class counterparts. At time
$t>0$, let $N_{cd}(t)$ be the number of vertices of class $d$ inside a
tree, the root of which is of class $c$. Then, as in proof of
theorem~\ref{thm.lawN}, if $z_{c}$ is a complex number such that
$|z_{c}|<1$, $\forall c\in\C$,
\[
 \EE\Bigl[\prod_{d\in\C}z_{d}^{N_{cd}(t)}\Bigr]
   =  z_c \exp\biggl\{\sum_{c'\in\C}\lambda_{cc'} 
                \EE \Bigl[\int_0^t 
                    \Bigl(\prod_{d\in\C}z_{d}^{N_{c'd}(x)}-1\Bigr)
                     \ind{x\leq\tau_{cc'}}dx\Bigr]\biggr\}
\]

Assume  the system is ergodic, let $N_{cd}\egaldef
\lim_{t\to\infty}N_{cd}(t)$ and 
\[
 \phi_c(\vec{z})\egaldef \EE\Bigl[\prod_{d\in\C}z_{d}^{N_{cd}}\Bigr].
\]

Then computations similar to the ones in theorem~\ref{thm.lawN} yield
\begin{equation}\label{eq.lawNc}
 \phi_c(\vec{z})=z_c\exp\biggl\{\sum_{c'\in\C}\lambda_{cc'}m_{cc'}
    (\phi_{c'}(\vec{z})-1)\biggr\},\ c\in\C.
\end{equation}

Unfortunately, no closed form solution is known for $\phi_c$ from this
equation. It is however possible, as for (\ref{eq.meanN}), to write
down a system of equations for the expectations of the $N_{cd}$'s.
\[
  \EE \bigl[N_{cd}\bigr] 
  = \ind{c=d}+\sum_{c'\in\C}\rho_{cc'}\exp\{r_{c'}\}\EE\bigl[N_{c'd}\bigr].
\]

This system admits of a non-negative matrix solution if, and only if,
the Perron-Frobenius eigenvalue of the matrix
\[
  M\egaldef\Bigl(\rho_{cc'}\exp\{r_{c'}\}\Bigr)_{c, c'\in\C}
\]
is smaller than $1$. A simple necessary condition for this to hold is
(\ref{eq.ergomulti.suf}).

Finally, the same line of argument allows to extend (\ref{eq.lawH}). If $H_c$
is the height in stationary regime of a tree which root is of class
$c\in\C$, then
\begin{eqnarray*}
\PP\{H_c= 0\} 
  & = &  e^{-r_c} , \\
\PP\{H_c> h+1\} 
  & = & 1-\exp\biggl[-\sum_{c'\in\C}\rho_{cc'}\exp\{r_{c'}\}
                                     \PP\{H_{c'}> h\}\biggr],
        \quad \forall h \ge 0 \,.
\end{eqnarray*}

\begin{acknowledgement}
The authors thank V.A.~Malyshev for bringing the single-class problem
to their attention and Th.~Deneux for skillful and useful numerical
experiments. They also want to thank the anonymous referee for his
(her) careful reading of the manuscript.
\end{acknowledgement}

\end{document}